\documentclass[12pt]{amsart}
\usepackage{mathscinet}
\usepackage{graphicx}
\usepackage{comment}
\usepackage[showonlyrefs]{mathtools}
\mathtoolsset{showonlyrefs}
\def\R{\mathbb R}
\def\C{\mathbb C}

\def\N{\mathbb N}
\def\EL{\mathcal B}
\def\ann{\operatorname{ann}}
\def\im{\operatorname{Im}}
\def\re{\operatorname{Re}}
\def\arg{\operatorname{arg}}
\def\meas{\operatorname{meas}}
\def\dens{\operatorname{dens}}
\newtheorem{theorem}{Theorem}[section]
\newtheorem{lemma}[theorem]{Lemma}
\theoremstyle{remark}
\newtheorem*{remark}{Remark}
\newtheorem*{acknowledgement}{Acknowledgement}
\numberwithin{equation}{section}

\title[Lebesgue measure of escaping sets]{Lebesgue measure of escaping sets of entire functions of completely
regular growth}
\author{Walter Bergweiler}
\address{Mathematisches Seminar, Christian-Albrechts-Universit\"at zu Kiel,
24098 Kiel, Germany}
\email{bergweiler@math.uni-kiel.de}
\author{Igor Chyzhykov}
\address{Faculty of Mechanics and Mathematics, Ivan Franko National University of Lviv,
Universytets'ka 1, 79000, Lviv, Ukraine}
\email{chyzhykov@yahoo.com}
\date{}
\begin{document}
\begin{abstract}
We give conditions ensuring that the Julia set and the escaping set of an entire function
of completely regular growth have positive Lebesgue measure. The essential
hypotheses are that the indicator is positive except perhaps at isolated points
and that most zeros are located in neighborhoods of finitely many rays.
We apply the result to solutions of linear differential equations.
\end{abstract}
\maketitle
\section{Introduction and results}
The \emph{Fatou set} of an entire function $f$ is defined as the set of all points in $\C$
where the iterates $f^n$ of $f$ form a normal family and
the \emph{Julia set} $J(f)$ is its complement.
The \emph{escaping set} $I(f)$ consists of all points which tend to $\infty$ under iteration.
By a result of Eremenko~\cite{Eremenko1989} we have $J(f)=\partial I(f)$.
For an introduction to the iteration theory of transcendental entire functions we 
refer to~\cite{Bergweiler1993} and~\cite{Schleicher2010}.

There has been considerable interest in the Hausdorff dimension and the Le\-besgue measure
of Julia sets and escaping sets.
For transcendental entire functions, the first results in this direction were obtained by
McMullen~\cite{McMullen1987} who proved that $J(\lambda e^z)$ has Hausdorff dimension $2$ 
and that $J(\sin(\alpha z+\beta))$ has positive Le\-besgue measure.
Here $\alpha,\beta,\lambda\in\C$, with $\alpha,\lambda\neq 0$.
These results have initiated a large body of research. Many of the subsequent results 
are concerned with the \emph{Eremenko-Lyubich class} $\EL$ of those entire functions for 
which the set of critical and (finite) asymptotic values is bounded.
Eremenko and Lyubich~\cite[Theorem~1]{Eremenko1992} proved that if $f\in\EL$, then 
$I(f)\subset J(f)$ and thus $J(f)=\overline{I(f)}$.
Note that the functions $\lambda e^z$ and $\sin(\alpha z+\beta)$ considered by McMullen
are in~$\EL$. McMullen actually proved that his results hold with $J(\cdot)$ replaced
by $I(\cdot)$.

Among the many results concerned with the Hausdorff dimension of Julia sets and escaping sets
we mention those of Bara\'nski~\cite{Baranski2008} and Schubert~\cite{Schubert2007}. They
show that for functions in $\EL$ of finite order the Julia set has Hausdorff dimension~$2$.
As the exponential function has order~$1$, this generalizes McMullen's result concerning
the functions~$\lambda e^z$. For further results concerning the Hausdorff dimension of Julia sets
we refer to the survey~\cite{Stallard2008}, as well as~\cite{Baranski2009,Bergweiler2010,Bergweiler2009,Bergweiler2008,Rempe2009,Rempe2010,Sixsmith2015}
for some more recent results

There are fewer results concerned with extending McMullen's result on the Le\-besgue
measure of $J(\sin(\alpha z+\beta))$ to larger classes of functions.
We mention~\cite{Aspenberg2012} where it is shown that
if $f\in \EL$ satisfies a certain growth condition, then $I(f)$ and hence $J(f)$ 
have positive measure.
Sixsmith~\cite{Sixsmith} showed that if $f$ has the form
\begin{equation} \label{intro1}
f(z)=\sum_{k=0}^{n-1}a_k\exp\!\left(\omega_n^k z\right), 
\end{equation}
where $n\geq 2$, $\omega_n$ is an $n$-th root of unity and $a_0\dots,a_{n-1}\in\C\backslash\{0\}$,
then $J(f)$ and $I(f)$ have positive measure. 
The functions  $\sin(\alpha z+\beta)$ are conjugated to functions of the form~\eqref{intro1}
with $n=2$,
concerning the latter functions.
Note that the functions given by~\eqref{intro1} are not in $\EL$ if $n\geq 3$.
Sixsmith noted that his method works more generally for exponential sums of the form
\begin{equation} \label{intro2}
f(z)=\sum_{k=0}^{n-1}a_k\exp\!\left(b_k z\right), 
\end{equation}
provided $\arg b_k<\arg b_{k+1} < \arg b_k +\pi$ for $0\leq k\leq n-2$ and
$\arg b_0<\arg b_{n-1}-\pi$, with the argument chosen in $[0,2\pi)$.

We shall consider more general classes of functions. 
In order to motivate and formulate the results, recall that if $\rho(r)$ is 
positive and differentiable for large $r$ and 
satisfies $\lim_{r\to\infty}\rho(r)=\rho\in (0,\infty)$ and $\lim_{r\to\infty}\rho'(r)r\log r=0$,
then $\rho(r)$ is called a \emph{proximate order}.
The \emph{indicator} $h$ of an entire function $f$ of order $\rho$ with respect to the
proximate order $\rho(r)$ is defined by
\[
h(\theta)=\limsup_{r\to\infty}\frac{\log|f(re^{i\theta})|}{r^{\rho(r)}}.
\]
The function $f$ is said to be of \emph{completely regular growth} (in the sense of Levin and 
Pfluger) if there exists disks $D(a_k,s_k)$ satisfying
\begin{equation}\label{C0}
\sum_{|a_k|\leq r} s_k=o(r) 
\end{equation}
such that
\begin{equation}\label{def-crg}
\log|f(re^{i\theta})| = h(\theta) r^{\rho(r)} +o(r^{\rho(r)})
\quad\text{for}\ 
re^{i\theta}\notin \bigcup_k D(a_k,s_k)
\end{equation}
as $r\to \infty$,
uniformly in~$\theta$.
A union of disks satisfying~\eqref{C0} is called a $C_0$-set.
We refer to Levin's book~\cite{Levin1964} for a thorough discussion of functions
of completely regular growth.

It follows from a result of Eremenko and Lyubich~\cite[Theorem~7]{Eremenko1992} that if a function
$f\in\EL$ is bounded in a sector, then $I(f)$ has measure zero.
Moreover, if -- in addition -- the singularities of the inverse are contained in the basins of attracting (or
parabolic) periodic points, then $J(f)$ also has measure zero~\cite[Theorem~8]{Eremenko1992}.

Thus in order to find conditions which imply that the escaping set and the Julia set 
of an entire function have positive measure, it seems reasonable to assume that 
the indicator is positive, except perhaps at isolated points.

We note that the exponential sum $f$ given by~\eqref{intro2} is of completely regular growth
with respect to the constant proximate order~$\rho(r)\equiv 1$, with indicator
\begin{equation}\label{indicator2}
h(\theta)=\max_{0\leq k\leq n-1} |b_k|\cos(\theta+\arg b_k).
\end{equation}
The condition on the arguments of the $b_k$ stated after~\eqref{intro2} ensures
that $h$ is positive. It is well-known (see, e.g., \cite{Polya1920,Schwengeler1925})
that the zeros of exponential sums are close to certain rays.

Our first result deals with entire functions for which all zeros are in the
neighborhood of certain rays and for which~\eqref{def-crg} holds with an error 
slightly better term than just $o(r^{\rho(r)})$.
\begin{theorem} \label{thm1.2}
Let 
$\theta_1<\theta_2<\dots <\theta_m<\theta_{m+1}=\theta_1+2\pi$ and
let $h\colon\R\to\R$ be a continuous function satisfying
$h(\theta)>0$ for $\theta_j<\theta<\theta_{j+1}$ and $j=1,\dots,m$.
Let $f$ be an entire function, $\rho(r)$  a proximate order and $\varepsilon(r)=1/\log^N(r)$ 
for some $N\in\N$, where $\log^N$ denotes the $N$-th iterate of the logarithm.
Suppose that
\begin{equation}\label{0a}
\log|f(re^{i\theta})|=h(\theta)r^\rho +O\!\left(r^{\rho(r)} \varepsilon(r)\right)
\end{equation}
whenever
\begin{equation}\label{0b}
|\theta-\theta_j|>\varepsilon(r) \quad \text{for } j=1,\dots,m.
\end{equation}
Then $I(f)\cap J(f)$ has positive measure.
\end{theorem} 
We will see in the proof that the hypotheses imply that 
the zeros of $f$ are contained in certain small neighborhoods of the $m$ rays
$\{z\colon \arg z=\theta_j\}$.
However, we may also allow zeros outside these neighborhoods
as long as their counting function is small.

We denote by $M(r,f)=\max_{|z|=r}|f(z)|$  the maximum modulus of an entire function~$f$.
Later, we will also use other standard terminology from the theory of entire
and meromorphic function~\cite{Goldberg2008}.
In particular, $T(r,f)$ denotes the Nevanlinna characteristic and
for $a\in\C$ and $r>0$ we denote by $n(r,a)=n(r,a,f)$ the number of $a$-points of
$f$ in the closed disk of radius $r$ around~$0$.
\begin{theorem} \label{thm1.3}
Let $f$ be an entire function satisfying the hypothesis of Theorem~\ref{thm1.2}
and let $g$ be an entire function satisfying
\[
\log M(r,g)=O\!\left( r^{\rho(r)}/\log^L r \right)
\]
for some $L\in\N$. Let $F=fg$.
Then $I(F)\cap J(F)$ has positive measure.
\end{theorem} 
As an application of these results we consider solutions of certain differential
equations.
Steinmetz (\cite{Steinmetz1989}, see also~\cite[\S 2]{Steinmetz1993}) showed 
that transcendental entire solutions of linear differential equations
\begin{equation}\label{lde}
w^{(n)}+p_{n-1}w^{(n-1)}+\dots + p_1w'+p_0w=0
\end{equation}
with rational coefficients $p_0,p_1,\dots,p_{n-1}$ are of completely regular growth
with respect to a constant proximate order $\rho(r)\equiv \rho>0$ and that
``most'' of the zeros are ``close'' to finitely many rays;
see Lemma~\ref{thm-steinmetz} below for a precise statement of Steinmetz's theorem.
We will use his results to prove the following.
\begin{theorem} \label{thm1.1}
Let $w$ be an entire function which solves a linear differential equation~\eqref{lde}
with rational coefficients~$p_0,p_1,\dots,p_{n-1}$. Suppose that the indicator of $w$ is positive
except possibly at isolated points. 
Then $I(w)\cap J(w)$ has positive measure.
\end{theorem} 
One consequence of this result is that Sixsmith's result \cite{Sixsmith} concerning 
functions of the form~\eqref{intro1} or~\eqref{intro2} remains valid if the $a_k$ occurring
there are polynomials instead of constants. Indeed, such functions satisfy a linear 
differential equation with constant coefficients.
Moreover, instead of $\arg b_{k+1} < \arg b_k +\pi$ and $\arg b_0<\arg b_{n-1}-\pi$ 
it suffices to assume that $\arg b_{k+1}\leq \arg b_k +\pi$ and $\arg b_0\leq\arg b_{n-1}-\pi$.

However, one may actually deduce from Theorem~\ref{thm1.3}
that  Sixsmith's result even holds if the $a_k$ are entire functions of order less than~$1$. We omit the details
of this argument.

We will deduce Theorems~\ref{thm1.2} and~\ref{thm1.3} from a general result saying that the escaping set
of an entire function $f$ has positive measure if both $f(z)$ and the real part of $zf'(z)/f(z)$ 
are large outside a set of small density.
For measurable subsets $X,Y$ of~$\C$, with $Y$ of non-zero Lebesgue measure $\meas(Y)$,
the \emph{density} $\dens(X,Y)$ of $X$ in~$Y$ is defined by
\[
\dens(X,Y)=\frac{\meas(X\cap Y)}{\meas(Y)}.
\]

Let $\beta\colon (0,\infty)\to (0,\infty)$ be a continuous,
increasing function which satisfies $\beta(x)>x$ for large $x$, say $x>x_0$. 
Then $\beta^n(x)\to\infty$ as $n\to\infty$ if $x>x_0$.

Let now $f$ be a transcendental entire function
and consider the sets
$$
A=A(f,\beta)=\left\{z\in\C\colon \re\!\left(\frac{zf'(z)}{f(z)}\right)>64   \text{ and } |f(z)|> \beta(|z|)\right\}
$$ 
and 
$$
B=B(f,\beta)=\left\{z\in A \colon  \re\!\left(\frac{\zeta f'(\zeta)}{f(\zeta)}\right)>0 \text{ for }
|\zeta -z|<32\left|\frac{f(z)}{f'(z)}\right|\right\}.
$$
\begin{theorem} \label{thm1.4}
Let $f$, $\beta$ and $B=B(f,\beta)$ be as above.
Let $\alpha\colon(0,\infty)\to(0,\infty)$ be a continuous, decreasing function which
satisfies $\alpha(r)\to 0$ as $r\to\infty$ and
\begin{equation}\label{1.4a}
\sum_{n=0}^\infty \alpha(\beta^{n}(r_0))<\infty
\end{equation}
for some $r_0>0$.
Suppose that 
\begin{equation}\label{1.4b}
\dens(B,\{z\colon r/2<|z|<2r\})\geq 1-\alpha(r)
\end{equation}
for all large~$r$.
Then $I(f)$ has positive measure.
Moreover, if $f$ does not have multiply-connected wandering domains, then
$I(f)\cap J(f)$ has positive measure.
\end{theorem}
The proof will show that if $r_0$ is large, then the set 
of all $z$ such $|z|>r_0$ and $|f^n(z)|>\beta^n(r_0)$ for all $n\in\N$ has positive measure,
and if $f$ does not have multiply-connected wandering domains, then the intersection of
this set with $J(f)$ also has positive measure.

It follows that if $\beta(r)\geq \exp r^\mu$ for some $\mu>0$ and all large~$r$,
then the conclusion of Theorem~\ref{thm1.4} holds with $I(f)$ replaced by the 
\emph{fast escaping set} $A(f)$ introduced in~\cite{Bergweiler1999}.
This set now plays an important role in transcendental dynamics; see, e.g., \cite{Rippon2005a,Rippon2012}.
In the proofs of Theorem~\ref{thm1.2}--\ref{thm1.1} we will actually have $\beta(r)=\exp r^{\rho-o(1)}$
so that $I(f)$ may be replaced by $A(f)$ in these results.

In order to apply Theorem~\ref{thm1.4} to a function $f$ we need estimates for 
the real part of $zf'(z)/f(z)$. 
For functions $f$ of completely regular growth, such estimates were obtained by
Gol'dberg and Korenkov~\cite{Goldberg1978,Goldberg1981}
as well as Gol'dberg, Sodin and Strochik~\cite{Goldberg1992}; see Section~\ref{further}
for a discussion of their results. However, for our applications these results are 
not sufficient as we need explicit estimates for the error terms arising there.
Our proofs of Theorems~\ref{thm1.2}--\ref{thm1.1} are independent of the
results in~\cite{Goldberg1978,Goldberg1981,Goldberg1992}.

A basic result of the theory of functions of completely regular growth says that the 
defining condition~\eqref{def-crg} is equivalent to a condition saying that the zeros are
regularly distributed in some sense; cf.~\cite[Chapters II and III]{Levin1964}. 
Assuming conditions more precise than~\eqref{def-crg} corresponds to 
stronger regularity conditions for the distribution of the zeros; see, e.g., \cite{Agranovich,VynKh04,VynKh05}.
We give a regularity condition for the distribution of zeros which implies
that the hypotheses of Theorem~\ref{thm1.2} hold.
Here we restrict to the case that the order is not an integer, although results for 
integer order could also be achieved. (The distinction between integer and non-integer 
order is standard in the theory; cf.~\cite[Sections~II.2 and~II.3]{Levin1964}.)
\begin{theorem} \label{thm1.5}
Let $f$ be an entire function of order $\rho\in (0,\infty)\backslash\N$.
Let $L\in\N$ and $\varepsilon(r)=1/\log^L(r)$.
Suppose that all but finitely many zeros of $f$ are contained in
$\{re^{i\theta} \colon  r>\exp^L(1), |\theta|<\varepsilon(r)\}$
and that
\begin{equation} \label{e:asymp_ray}
 | n(r, 0)- c r^{\rho(r)}|=
O\!\left(\varepsilon(r)r^{\rho(r)}\right)
\end{equation}
for some $c>0$ and some proximate order $\rho(r)$ satisfying  
\begin{equation}\label{e:0x}
\rho'(r) r\log r =O(\varepsilon(r))
\end{equation}
as $r\to\infty$.
Then
\begin{equation}\label{e:0a}
\log|f(re^{i\theta})|=
c\frac{\pi\cos((\theta-\pi)\rho(r)) }{\sin \pi \rho(r)} r^{\rho(r)}
+ O\!\left(\sqrt[4]{\varepsilon(r)}r^{\rho(r)}\right)
\end{equation}
for
\begin{equation}\label{e:0b}
\sqrt{\varepsilon(r)}   \leq \theta\leq 2\pi- \sqrt{\varepsilon(r)}.
\end{equation}
\end{theorem}
We remark that standard constructions of proximate orders as in~\cite{Levin1964} yield $|\rho'(r) r\log r|\leq 1/\log\log r$.
Thus~\eqref{e:0x} is not very restrictive.

We also note that if $\rho(r)-\rho=O\!\left(\sqrt[4]{\varepsilon(r)}\right)$, then~\eqref{e:0a} simplifies
to
\begin{equation}\label{e:0c}
\log|f(re^{i\theta})|=h(\theta) r^{\rho(r)}+O\!\left(\sqrt[4]{\varepsilon(r)}r^{\rho(r)}\right)
\end{equation}
where
\begin{equation}\label{e:0d}
h(\theta)=c\frac{\pi \cos \rho(\theta-\pi)}{\sin \pi \rho}.
\end{equation}
If $0<\rho\leq 1/2$, then we have $h(\theta)>0$ for $0<\theta<2\pi$
so that the hypotheses of Theorem~\ref{thm1.2} are satisfied (for $N=L+1$).

Theorem~\ref{thm1.5} immediately gives a corresponding result for the case
that the zeros are contained in finitely
many domains $S_j=\{re^{i\theta} \colon |\theta-\theta_j|<\varepsilon(r)\}$
and their number $n_j(r,0)$ in $S_j\cap \{z\colon |z|\leq r\}$ satisfies 
$| n_j(r, 0)- c_j r^{\rho(r)}|=O\!\left(\varepsilon(r)r^{\rho(r)}\right)$
for certain $c_j>0$.
\section{Proof of Theorem \ref{thm1.4}}
We first state a number of lemmas that we require. Here and in the following we denote 
for $a\in\C$ and $r>0$ by $D(a,r)=\{z\colon |z-a|<r\}$ the disk of radius $r$ around $a$.

The Koebe distortion theorem and the Koebe one quarter
theorem are usually stated for functions univalent in the unit disk $D(0,1)$.
The following lemma is deduced by a simple transformation from this.
\begin{lemma}\label{koebe1}
Let $f$ be univalent in  $D(a,r)$ and let 
$0<\rho<1$.
Then
\[
\frac{\rho}{(1+\rho)^2}
|f'(a)|
\leq \left| \frac{f(z)-f(a)}{z-a} \right|
\leq
\frac{\rho}{(1-\rho)^2}
|f'(a)|
\quad\text{for } |z-a|=\rho r
\]
and
\[
\frac{1-\rho}{(1+\rho)^3}
|f'(a)|
\leq |f'(z)|
\leq
\frac{1+\rho}{(1-\rho)^3}
|f'(a)|
\quad\text{for } |z-a|\leq \rho r.
\]
Moreover,
\[
f(D(a,r))\supset D(f(a),r|f'(a)|/4).
\]
\end{lemma}
The following lemma is a direct consequence, but it can also easily be proved directly by a simple compactness argument.
\begin{lemma}\label{koebe2}
Let $\Omega$ be a domain in $\C$ and let $Q$ be a compact subset of $\Omega$.
Then there exists a positive constant $C$ such that if $f$ is univalent in $\Omega$ and
$z,\zeta\in Q$, then $|f'(\zeta)|\leq C|f'(z)|$.
\end{lemma}
The constant $C$ in Lemma~\ref{koebe2} is also called the distortion of $f$ on~$Q$.
The following result says that functions of distortion $C$  increase the density of
sets at most by a factor $C^2$.
\begin{lemma}\label{dist}
Let $f$ be univalent in a domain containing measurable sets $P$ and~$Q$.
Suppose that there exists a positive constant $C$ such that 
$|f'(\zeta)|\leq C|f'(z)|$ for all $\zeta,z\in Q$. Then 
$\dens(f(P),f(Q))\leq C^2 \dens(P,Q)$.
\end{lemma}
\begin{proof}
We have 
\begin{align*}
&\quad\ \meas(f(P)\cap f(Q))
= \meas(f(P\cap Q))
\\ &
=
\int_{P\cap Q}|f'(z)|^2dx\,dy
\leq
\meas(P\cap Q)\sup_{z\in Q}|f'(z)|^2
\\ &
\leq
C^2\meas(P\cap Q)\inf_{z\in Q}|f'(z)|^2
=
C^2 \dens(P,Q) \meas(Q) \inf_{z\in Q}|f'(z)|^2
\\ &
\leq
C^2 \dens(P,Q) 
\int_{Q}|f'(z)|^2dx\,dy
=C^2 \dens(P,Q) \meas(f(Q)),
\end{align*}
from which the conclusion follows.
\end{proof}
The following lemma is standard~\cite[Proposition~1.10]{Pommerenke1992}.
\begin{lemma}\label{prp}
Let $\Omega$ be a convex domain and let $f\colon \Omega\to\C$ be holomorphic.
If $\re f'(z)>0$ for all $z\in\Omega$, then $f$ is univalent.
\end{lemma}
The next lemma characterizes convex univalent functions~\cite[Theorem~2.11]{Duren1983}.
\begin{lemma}\label{dur}
Let $f\colon D(0,1)\to \C$ be holomorphic. 
Then $f$ maps $D(0,1)$ univalently onto a convex domain if and only if
$\re (1+zf''(z)/f'(z))>0$ for all $z\in D(0,1)$.
\end{lemma}
The following result is known as the Besicovitch covering theorem~\cite[Theorem~3.2.1]{deGuzman1981}. 
Here we use our notation $D(a,r)$ for disks in $\C$  also for balls in~$\R^n$.
\begin{lemma} \label{besicovitch}
Let $K\subset \R^n$ be bounded and
$r\colon K\to(0,\infty)$. Then there exists an at most countable subset $L$
of $K$
satisfying
\[
K\subset \bigcup_{x\in L} D(x,r(x))
\]
such that no point in $\R^n$ is contained in more than $4^{2n}$
of the balls $D(x,r(x))$, $x\in L$.
\end{lemma}

\begin{proof}[Proof of Theorem~\ref{thm1.4}]
We denote the annulus $\{z\colon r/2<|z|<2r\}$ appearing in the hypothesis~\eqref{1.4b} by~$\ann(r)$. 
We use the abbreviation $L(z)=f'(z)/f(z)$ for the logarithmic derivative.
For $z\in B$ we will frequently consider the disks 
\[
\Delta_1(z)=D(z,1/(8|L(z)|))
\quad  \text{and} \quad
\Delta_2(z)=D(z,32/|L(z)|).
\]
By the definition of $B$ we have $\re(\zeta f'(\zeta)/f(\zeta))>0$ and thus in particular
$f(\zeta)\neq 0$ for $\zeta\in \Delta_2(z)$.
We also put $X=\C\backslash B$.

First we show that
there exists a constant $c_1$ such that if 
$z\in B$, with $|z|$ sufficiently large, then 
\begin{equation}\label{1a}
\dens(f^{-1}(X),\Delta_1(z))\leq c_1\alpha(|f(z)|).
\end{equation}
Since $f$ has no zeros in $\Delta_2(z)$, we can define
a branch $\varphi$ of $\log f$ in $\Delta_2(z)$.
We will show that $\varphi$ is univalent.
In order to do so,
let $W$ be a component of $\exp^{-1}(\Delta_2(z))$
and put $F\colon W\to\C$, $F(\zeta)=\varphi(e^\zeta)=\log f(e^\zeta)$.
Then 
\[
\re F'(\zeta)=\re (e^\zeta L(e^\zeta))>0
\]
for $\zeta\in W$.
Applying Lemma~\ref{dur} to $f(\zeta)=\log ((1+r\zeta)z)$, with the appropriate branch of the logarithm
and $r=32/|zL(z)|<1/2<1$,
we see that 
$W$ is convex. Thus $F$ and hence $\varphi$ are univalent by Lemma~\ref{prp}.

By Koebe's one quarter theorem (Lemma~\ref{koebe1}),
\[
\varphi(\Delta_2(z))= \varphi(D(z,32/|L(z)|)
\supset D(\varphi(z),8|\varphi'(z)|/|L(z)|)=D(\varphi(z),8).
\]
Let 
$$
Q=\{\zeta\in\C\colon  |\re(\zeta-\varphi(z))|\leq \log 2,|\im(\zeta-\varphi(z))|\leq \pi\}.
$$
Thus $Q$ is a rectangle centered at $\varphi(z)$.
Since  $Q\subset D(\varphi(z),4)$, we have 
\begin{equation}\label{1b}
f(\Delta_2(z))\supset \exp(D(\varphi(z),8))\supset \exp(Q)=\ann(|f(z)|)).
\end{equation}

By hypothesis, 
\[
\dens(X,\ann(|f(z)|)) \leq \alpha(|f(z)|).
\] 
As the distortion of the exponential map
on (the interior of) the rectangle $Q$ is bounded by~$4$, we obtain
\begin{equation}\label{1b1}
\dens(\exp^{-1}(X),Q)
\leq 16 \dens(X,\exp(Q))\leq 16 \alpha(|f(z)|).
\end{equation}
Also, since
$$
Q\subset D(\varphi(z),4)\subset D(\varphi(z),8)\subset  \varphi(\Delta_2(z)),
$$
Koebe's theorem, applied with $\rho=1/2$, yields that
the distortion of the inverse of $\varphi\colon \Delta_2(z)\to  \varphi(\Delta_2(z))$ on $Q$
is bounded by $(1+\rho)^4/(1-\rho)^4=81$.
Thus
\begin{equation}\label{1b2}
\begin{aligned}
\dens(f^{-1}(X),\varphi^{-1}(Q))
&=\dens(\varphi^{-1}(\exp^{-1}(X)),\varphi^{-1}(Q))) 
\\ &
\leq 81^2 \dens(\exp^{-1}(X),Q).
\end{aligned}
\end{equation}
Finally, $\varphi^{-1}(Q)$ is not a disk, but on the one hand we have 
$\varphi^{-1}(Q)\subset \Delta_2(z)$, while on the other hand 
Koebe's theorem yields
\begin{align*}
\varphi^{-1}(Q)
&\supset \varphi^{-1}(D(\varphi(z),\log 2))
\supset 
D(z,|(\varphi^{-1})'(\varphi(z))|(\log 2)/4)
\\ &
=
D(z,(\log 2)/(4|L(z)|)
\supset
D(z,1/(8|L(z)|)
=\Delta_1(z).
\end{align*}
We conclude that
\begin{equation}\label{1b3}
\begin{aligned}
&\quad\ \dens(f^{-1}(X),\Delta_1(z))
= \frac{\meas (f^{-1}(X)\cap \Delta_1(z))}{\meas \Delta_1(z)}
\\ &
\leq 
\frac{\meas (f^{-1}(X)\cap \varphi^{-1}(Q))}{\meas \Delta_1(z)}
=2^{16} \frac{\meas (f^{-1}(X)\cap \varphi^{-1}(Q))}{\meas \Delta_2(z)}
\\ &
\leq 
2^{16} \frac{\meas (f^{-1}(X)\cap \varphi^{-1}(Q))}{\meas \varphi^{-1}(Q)}
=2^{16}\dens(f^{-1}(X),\varphi^{-1}(Q)).
\end{aligned}
\end{equation}
Combining~\eqref{1b1}, \eqref{1b2} and \eqref{1b3}
we obtain \eqref{1a} with $c_1=2^{16}\cdot 81^2\cdot 16$.

A slight variation of the argument used to prove \eqref{1b} 
shows that if $z\in B$ is sufficiently
large, then there exists a subdomain of $\Delta_2(z)$ which is mapped
univalently onto $D(f(z),|f(z)|/2)$.
If also $f(z)\in B$, then 
\[
|f(z)L(f(z))|\geq \re f(z)L(f(z))\geq 64
\]
and thus $D(f(z),|f(z)|/2)\supset \Delta_2(f(z))$.
Hence there is also a subdomain $U_1(z)$ of  $\Delta_2(z)$ such that
$f\colon U_1(z)\to \Delta_2(f(z))$ is univalent.

For $n\geq 0$ we now consider the set
\[T_n=\{z\in B\colon f^k(z)\in B\mbox{ for }0\leq k\leq n\}.\]
Thus $T_0=B$.
It follows that if $z\in T_n$, then there exists a neighborhood 
$U_n(z)$ of $z$ such that $f^n\colon U_n(z)\to \Delta_2(f^n(z))$ 
is univalent. 
Moreover, $U_n(z)$ contains a subdomain $V_n(z)$ such that 
$f^n\colon V_n(z)\to \Delta_1(f^n(z))$  is univalent. 
We apply the Koebe distortion theorem to the branch of the inverse of $f^n$ 
which maps $\Delta_2(f^n(z))$ to $U_n(z)$, with $\rho=2^{-8}$. 
Since $(1+\rho)^2/(1-\rho)^2\leq 2$
we see that there exists $r_n=r_n(z)$ such that 
$D(z,r_n)\subset V_n(z)\subset D(z,2 r_n)$.

Moreover, we deduce from Koebe's theorem
that the distortion of the inverse of 
$f^n$ is bounded by $(1+\rho)^3/(1-\rho)^3$ on $\Delta_1(f^n(z))$.
Since also $(1+\rho)^6/(1-\rho)^6\leq 2$ Lemma~\ref{dist} yields
$\dens (Y,V_n(z))\leq 2\dens(f^n (Y),\Delta_1(f^n(z))$
and hence we obtain
\[
\begin{aligned}
\dens(Y,D(z,r_n))
&=\frac{\meas (Y\cap D(z,r_n))}{\meas D(z,r_n)}
\leq 4\frac{\meas (Y\cap V_n(z))}{\meas D(z,2r_n)}
\\ &
\leq 4 \dens(Y,V_n(z))
\leq 8 \dens(f^n (Y),\Delta_1(f^n(z))
\end{aligned}
\]
for any measurable set~$Y$.
We apply this to $Y=T_n\backslash T_{n+1}$ and, in combination with \eqref{1a}, find that 
\[
\dens(T_n\backslash T_{n+1},D(z,r_n(z))) 
\leq 8 \dens(f^{-1} (X),\Delta_1(f^n(z)))
\leq 8c_1 \alpha(|f^{n+1}(z)|)
\]
for $z\in T_n$.
Since $|f^{n+1}(z)|\geq \beta^{n+1}(|z|)$ for  $z\in T_n$ we thus have
\begin{equation}\label{1c}
\dens(T_n\backslash T_{n+1},D(z,r_n(z)))
\leq 8c_1 \alpha(|\beta^{n+1}(|z|)|)
\end{equation}
for $z\in T_n$.

Let $R$ be large.
The Besicovitch covering theorem yields a countable 
subset $L_n$ of $\ann(R)$ such  the disks $\{D(z,r_n(z))\colon z\in L_n\}$ cover $B\cap \ann(R)$,
with no point covered more than $4^4$ times.
We may assume that all these disks $D(z,r_n(z))$ are contained in $D(0,3R)$.
Then 
\[
\begin{aligned}
 \meas (T_n\backslash T_{n+1}\cap \ann(R))
\leq &
\meas \left(T_n\backslash T_{n+1}\cap \bigcup_{z\in L_n} D(z,r_n(z)) \right)
\\ 
\leq &
\sum_{z\in L_n} \meas \left(T_n\backslash T_{n+1}\cap D(z,r_n(z)) \right)
\\ 
= &
\sum_{z\in L_n} \dens \left(T_n\backslash T_{n+1}, D(z,r_n(z)) \right) \meas D(z,r_n(z)).
\end{aligned}
\]
Using \eqref{1c} we conclude that
\[
\begin{aligned}
 \meas (T_n\backslash T_{n+1}\cap \ann(R))
&\leq 
8c_1 \sum_{z\in L_n}  \alpha(\beta^{n+1}(|z|))  \meas D(z,r_n(z))
\\
&\leq  
 8c_1 \alpha(\beta^{n+1}(R/2)) \sum_{z\in L_n}  \meas D(z,r_n(z))
\\
&\leq  
8c_1 \alpha(\beta^{n+1}(R/2))4^4 \meas D(0, 3R)
\\ 
&=
c_2 \alpha(\beta^{n+1}(R/2))
\end{aligned}
\]
with $c_2= 4^4\cdot 72 \cdot  c_1 \pi R^2$.
We put $T=\bigcap_{n=0}^\infty T_n$. Then
$T=T_0\backslash \bigcup_{n=0}^\infty(T_n\backslash T_{n+1})$ and thus
\[
\begin{aligned}
\meas (T\cap \ann(R)) 
& \geq
\meas (T_0\cap \ann(R))
 - \sum_{n=0}^\infty \meas (T_n\backslash T_{n+1}\cap \ann(R))
\\ & \geq
(1-\alpha(R))\meas  (\ann(R))-c_2\sum_{n=0}^\infty \alpha(\beta^{n+1}(R/2))
\\ & \geq
R^2 -c_2\sum_{n=0}^\infty \alpha(\beta^{n+1}(R/2)).
\end{aligned}
\]
It follows that $T$ and hence $I(f)$ have positive measure if $R$ is large enough.
\end{proof}
\section{Proof of Theorems~\ref{thm1.2}, \ref{thm1.3} and \ref{thm1.1}}
We recall the Schwarz integral formula which says that if $g$ is holomorphic
in a domain containing the closed disk $\overline{D(a,t)}$, then
\begin{equation}\label{schwarz0}
g(z)=\frac{1}{2\pi i} \int_{|\zeta-a|=t}\frac{\zeta+z}{\zeta-z}\re g(\zeta)\frac{d\zeta}{\zeta} +i\im g(a)
\end{equation}
for $z\in D(a,t)$. It follows that
\begin{equation}\label{schwarz1}
g'(z)=\frac{1}{\pi i} \int_{|\zeta-a|=t}\frac{\re g(\zeta)}{(\zeta-z)^2}d\zeta.
\end{equation}
We collect some properties of proximate orders that we need.
The first two properties can be found in~\cite[Section~I.12]{Levin1964}.
The third one is also certainly known, but we have not found it explicitly 
stated in the literature.
\begin{lemma} \label{lemma-prox}
Let $\rho(r)$ be a proximate order with $\rho=\lim_{r\to\infty}\rho(r)\in (0,\infty)$.
Put $V(r)=r^{\rho(r)}$.
Then
\begin{equation}\label{prox0}
V(2r)=O(V(r)),
\end{equation}
\begin{equation}\label{prox1}
V(s)=(1+o(1)) V(r)
\quad \text{as}\ 
r\to\infty,\ \frac{s}{r}\to 1.
\end{equation}
and
\begin{equation}\label{prox2}
V(s)=V(r)+\rho V(r)\left(\frac{s}{r}-1\right) +o\left( V(r) \left(\frac{s}{r}-1\right)\right)
\quad \text{as}\ 
r\to\infty,\ \frac{s}{r}\to 1.
\end{equation}
\end{lemma}
\begin{proof}
Since 
\begin{equation}\label{prox3}
\frac{d\log V(t)}{dt}= \rho'(t)\log t+\frac{\rho(t)}{t}
=\frac{\rho(t)}{t}\left(\frac{\rho'(t)t\log t}{\rho(t)}
+1\right)=(1+o(1))\frac{\rho}{t}
\end{equation}
by the definition of a proximate order, we have
\begin{equation}\label{prox4}
\log \frac{V(s)}{V(r)}= \int_r^s \left( \rho'(t)\log t+\frac{\rho(t)}{t}\right) dt
=(1+o(1))\rho\log \frac{s}{r}
\end{equation}
from which~\eqref{prox0} and~\eqref{prox1} follow.
Since $\log x=x-1+o(\log x)$ as $x\to 1$ we can also deduce from~\eqref{prox4} that
\begin{equation}\label{prox5}
\frac{V(s)}{V(r)}-1
=\rho\left( \frac{s}{r}-1\right) +o\left(\log \frac{s}{r}\right) 
=\rho\left( \frac{s}{r}-1\right) +o\left(\frac{s}{r}-1\right) 
\end{equation}
and hence
\begin{equation}\label{prox6}
V(s)-V(r)=\rho V(r)\left(\frac{s}{r}-1\right) + o\left( V(r) \left(\frac{s}{r}-1\right)\right),
\end{equation}
which is~\eqref{prox2}.
\end{proof}

We shall also need the following result of Zheng~\cite[Corollary~5]{Zheng2006}.
\begin{lemma} \label{lemma-zheng}
Let $f$ be an entire function.
Suppose that there exists $d>1$ such that
$\log M(2r,f)\geq d \log M(r,f)$ for all large~$r$.
Then the Fatou set of $f$ has no multiply connected components.
\end{lemma}
It follows that if $f$ has completely regular growth, then 
the Fatou set of $f$ has no multiply connected components.
\begin{proof}[Proof of Theorem~\ref{thm1.2}]
We put $\varepsilon_1(r)=\varepsilon(r)=1/\log^N r$,
$\varepsilon_2(r)=\sqrt{\varepsilon_1(r)}$ 
and
$\varepsilon_3(r)=\sqrt{\varepsilon_2(r)}$.
Then 
$\varepsilon_k(r)/\varepsilon_{k+1}(r)\to 0$
as $r\to\infty$, for $k=1,2$.

As already mentioned after the statement of the theorem, the hypotheses imply that the 
zeros are near the rays $\{z\colon \arg z=\theta_j\}$. More precisely, 
\eqref{0a} and~\eqref{0b} imply that for $1\leq j\leq m$
and $\eta>0$ there exists $r_\eta>0$ such that $f(re^{i\theta})\neq 0$ if
$\theta_j+\eta\leq\theta\leq \theta_{j+1}-\eta$ and $r\geq r_\eta$.
This implies (see~\cite[p.~115]{Levin1964}) that there exist $A_j,\varphi_j\in\R$ such that 
\begin{equation}\label{4a}
h(\theta)=A_j\cos(\rho\theta +\varphi_j) 
\quad\text{for} \ 
\theta_j\leq\theta\leq \theta_{j+1}.
\end{equation}
Since $h(\theta)>0$ for $\theta_j<\theta< \theta_{j+1}$ by hypothesis is follows from~\eqref{4a} 
that there exists $c_j>0$ such that 
\begin{equation}\label{4c}
h(\theta)\geq c_j \min\{\theta-\theta_j,\theta_{j+1}-\theta\}
\quad\text{for} \
\theta_j\leq\theta\leq \theta_{j+1}.
\end{equation}
We now conclude from~\eqref{0a} and~\eqref{4c} that if $p>1$ 
is sufficiently large
and $\theta_j+p\varepsilon_1(r) \leq\theta\leq \theta_{j+1}-p\varepsilon_1(r)$,
then
\begin{equation}\label{4c1}
\begin{aligned}
\log|f(re^{i\theta})|
&=
h(\theta)r^{\rho(r)} +O\!\left(r^{\rho(r)} \varepsilon_1(r)\right)
\\ &
\geq p c_j r^{\rho(r)} \varepsilon_1(r)-O\!\left(r^{\rho(r)} \varepsilon_1(r)\right)
\geq r^{\rho(r)} \varepsilon_1(r)
\end{aligned}
\end{equation}
for all large $r$, say $r\geq r_0$.
With
\[
U_j=\left\{re^{i\theta}\colon r\geq r_0\text{ and } \theta_j+p\varepsilon_1(r)\leq\theta\leq \theta_{j+1}-p\varepsilon_1(r)\right\}
\]
we thus have
\begin{equation}\label{4d}
\log |f(z)|\geq |z|^{\rho(|z|)} \varepsilon_1(|z|)
\quad\text{for} \ z\in U_j.
\end{equation}
In particular, we have $f(z)\neq 0$ for $z\in U_j$, provided $r_0$ was chosen large enough.
We may thus define a branch of $\log f$ in $U_j$. 
Let
\[
V_j=\left\{re^{i\theta}\colon r\geq r_1\text{ and }
\theta_j+3\varepsilon_2(r)\leq\theta\leq \theta_{j+1}-3\varepsilon_2(r)\right\}.
\]
Then $V_j\subset U_j$ for large~$r_1$. In fact, if $r_1$ is chosen large enough and $z=re^{i\theta}\in V_j$,
then the closed disk of radius
$t_r=r\varepsilon_2(r)$ around $z$ is contained in~$U_j$.
Since $\log|f|=\re(\log f)$ it thus follows from~\eqref{schwarz1}
with $\gamma(\varphi)=z+t_re^{i\varphi}$ that
\begin{equation}\label{8a}
\frac{f'(z)}{f(z)}
=\frac{1}{\pi i} \int_{|\zeta-z|=t_r}\frac{\log |f(\zeta)|}{(\zeta-z)^2}d\zeta
=\frac{1}{\pi t_r} \int_{-\pi}^\pi \log |f(\gamma(\varphi))| e^{-i\varphi}d\varphi.
\end{equation}
We write 
\begin{equation}\label{8b}
\log|f(z)|=P(z)+R(z)
\quad\text{with}\ 
P(re^{i\theta})=h(\theta)r^{\rho(r)} 
\end{equation}
so that
\begin{equation}\label{8c}
R(z)=O\!\left(|z|^{\rho(|z|)}\varepsilon_1(|z|)\right)
\quad\text{for}\ 
z\in U_j
\end{equation}
by hypothesis.
This implies that 
\begin{equation}\label{8d}
\frac{1}{\pi t_r} \int_{-\pi}^\pi R(\gamma(\varphi)) e^{-i\varphi}d\varphi
=O\!\left(\frac{r^{\rho(r)}\varepsilon_1(r)}{t_r}\right)
=O\!\left(\frac{r^{\rho(r)-1}\varepsilon_1(r)}{\varepsilon_2(r)}\right)
=O\!\left(r^{\rho(r)-1}\varepsilon_2(r)\right)
\end{equation}
so that
\begin{equation}\label{8e}
\begin{aligned}
\frac{zf'(z)}{f(z)}
&=\frac{re^{i\theta}}{\pi t_r} \int_{-\pi}^\pi P(\gamma(\varphi)) e^{-i\varphi}d\varphi
+O\!\left(r^{\rho(r)}\varepsilon_2(r)\right)
\\ &
=\frac{e^{i\theta}}{\pi \varepsilon_2(r)} \int_{-\pi}^\pi P(\gamma(\varphi)) e^{-i\varphi}d\varphi
+O\!\left(r^{\rho(r)}\varepsilon_2(r)\right)
\end{aligned}
\end{equation}
by~\eqref{8a} and~\eqref{8b}.
To evaluate the real part of the integral on the right side we note that
\begin{equation}\label{8f}
\begin{aligned}
\re\!\left( e^{i\theta} \int_{-\pi}^\pi P(\gamma(\varphi)) e^{-i\varphi} d\varphi \right)
&=
\re\!\left( e^{i\theta} \int_{-\theta-\pi}^{-\theta+\pi} P(\gamma(\theta+\psi))e^{-i(\theta+\psi)}d\psi \right)
\\ &
=
\re\!\left(\int_{-\pi}^{\pi} P(\gamma(\theta+\psi))e^{-i\psi} d\psi \right)
\\ &
=
\int_{-\pi}^{\pi} P(\gamma(\theta+\psi))\cos \psi \; d\psi.
\end{aligned}
\end{equation}
Now
$|\gamma(\theta+\psi)|=|re^{i\theta}+t_re^{i(\theta+\psi)}|
=r|1+\varepsilon_2(r) e^{i\psi}|= |\gamma(\theta-\psi)|$
and hence 
\begin{equation}\label{8f1}
V(|\gamma(\theta+\psi)|)
=V(|\gamma(\theta-\psi)|)
\end{equation}
while, with 
\begin{equation}\label{8f2}
\delta(\psi)=\arg(1+\varepsilon_2(r)e^{i\psi})
=\arg(re^{i\theta}+t_re^{i(\theta+\psi)})-\theta
=\arg\gamma(\theta+\psi) -\theta,
\end{equation}
we have $\delta(-\psi)=-\delta(\psi)$. Thus
\begin{equation}\label{8f3}
P(\gamma(\theta\pm\psi))=h(\theta\pm\delta(\psi))V(|\gamma(\theta+\psi)|)
\end{equation}
so that~\eqref{8f} yields
\begin{equation}\label{8g}
\begin{aligned}
\quad\ 
&
\re\!\left( e^{i\theta} \int_{-\pi}^\pi P(\gamma(\varphi)) e^{-i\varphi} d\varphi \right)
\\ 
=&
\int_{-\pi}^{\pi} h(\theta+\delta(\psi) V(|\gamma(\theta+\psi)|)\cos \psi \; d\psi.
\\ 
=&
\frac12 \int_{-\pi}^{\pi} (h(\theta+\delta(\psi))+ h(\theta-\delta(\psi) )) V(|\gamma(\theta+\psi)|)\cos \psi \; d\psi.
\end{aligned}
\end{equation}
Since 
$|\delta(\psi)|\leq \arcsin \varepsilon_2(r)=O(\varepsilon_2(r))$ and
$h(\theta+\delta)-2h(\theta)+h(\theta-\delta)=O(\delta^2)$ as $\delta\to 0$ by~\eqref{4a},
this yields
\begin{equation}\label{8h}
\begin{aligned}
\quad\ 
&
\re\!\left( e^{i\theta} \int_{-\pi}^\pi P(\gamma(\varphi)) e^{-i\varphi} d\varphi \right)
-h(\theta) \int_{-\pi}^{\pi}V(|\gamma(\theta+\psi)|)\cos \psi \; d\psi
\\ 
=&
\frac12 \int_{-\pi}^{\pi} (h(\theta+\delta(\psi))- 2h(\theta)+ h(\theta-\delta(\psi) )) V(|\gamma(\theta+\psi)|)\cos \psi \; d\psi
\\ 
=&
O\!\left(V(r)\varepsilon_2(r)^2\right).
\end{aligned}
\end{equation}
Using~\eqref{8e} we obtain
\begin{equation}\label{8i}
\re\frac{zf'(z)}{f(z)}
=
\frac{h(\theta)}{\pi\varepsilon_2(r)} \int_{-\pi}^{\pi}V(|\gamma(\theta+\psi)|)\cos \psi \; d\psi
+O\!\left(V(r)\varepsilon_2(r)\right).
\end{equation}
Noting that 
\begin{equation}\label{8j}
\begin{aligned}
\frac{|\gamma(\theta+\psi)|}{r}-1
&=\left|e^{i\theta}+\frac{t_r}{r}e^{i(\theta+\psi)}\right|-1
\\ &
=\left|1+\varepsilon_2(r)e^{i\psi}\right|-1
=\varepsilon_2(r)\cos\psi +O\!\left(\varepsilon_2(r)^2\right)
\end{aligned}
\end{equation}
we deduce from~\eqref{prox2} that
\begin{equation}\label{8k}
\begin{aligned}
&\quad \ 
\int_{-\pi}^{\pi}V(|\gamma(\theta+\psi)|)\cos \psi \; d\psi
\\ &
=  \rho V(r) \int_{-\pi}^{\pi}
\left(\frac{|\gamma(\theta+\psi)|}{r}-1\right)\cos\psi \; d\psi
+O\!\left(V(r)\varepsilon_2(r)^2\right)
\\ &
= \pi \rho V(r) \varepsilon_2(r)
+O\!\left(V(r)\varepsilon_2(r)^2\right)
\end{aligned}
\end{equation}
and hence~\eqref{8i} yields
\begin{equation}\label{8l}
\re\frac{zf'(z)}{f(z)}
=\rho h(\theta) V(r) +
O\!\left(V(r)\varepsilon_2(r)\right)
\quad\text{for}\ 
z\in V_j.
\end{equation}
We put
\begin{equation}\label{8m}
W_j=\left\{re^{i\theta}\colon r\geq r_2\text{ and }
\theta_j+ \varepsilon_3(r)\leq\theta\leq \theta_{j+1}- \varepsilon_3(r)\right\}.
\end{equation}
It follows from~\eqref{4c} and~\eqref{8l} that
\begin{equation}\label{4h}
\begin{aligned}
\re \frac{zf'(z)}{f(z)} 
&\geq (1-o(1))c_j \rho |z|^\rho \varepsilon_2(|z|) 
\\ &
\geq \frac12 c_j \rho |z|^\rho \varepsilon_2(|z|) \geq 64 
\quad\text{for} \ z\in W_j,
\end{aligned}
\end{equation}
provided $r_2$ was chosen large enough.

Put $\beta(r)=\exp(r^{\rho(r)}\varepsilon_1(r))$ and let $A=A(f,\beta)$ and $B=B(f,\beta)$ be the sets 
introduced before the statement of Theorem~\ref{thm1.4}.
Since $U_j\supset V_j\supset W_j$ it follows from~\eqref{4d} and~\eqref{4h} that $A$ contains~$W_j$.
Let $r_3>r_2$ and put
\begin{equation}\label{8o}
X_j=\left\{re^{i\theta}\colon r\geq r_3\text{ and }
\theta_j+3 \varepsilon_3(r)\leq\theta\leq \theta_{j+1}-3 \varepsilon_3(r)\right\}.
\end{equation}
Then $X_j\subset W_j\subset A$. In fact, if $z\in W_j$, then the disk of radius 
$|z|\varepsilon_3(|z|)$ around $z$
is contained in~$V_j$, provided $r_3$ is chosen sufficiently large.
Moreover, if $z\in W_j$ and $|z|$ is sufficiently large, then
\[
\frac{32|f(z)|}{|f'(z)|} \leq \frac{32|z|}{\displaystyle \re\!\left(\frac{zf'(z)}{f(z)}\right)}
\leq \frac{64|z|}{c_j \rho |z|^{\rho(|z|)} \varepsilon_2(|z|)}
\leq |z|  \varepsilon_3(|z|) 
\]
by~\eqref{4h}. Assuming that $r_2$ has been chosen large enough
we deduce that $B$ contains the $X_j$; that is, with $X=\bigcup_{j=1}^m X_j$ we
have $X\subset B$.

For large $r>2r_2$ we have 
\begin{equation}\label{4i}
\dens(X,\ann(r))\geq 
1-6m\,\varepsilon_3(r/2).
\end{equation}
Since $X\subset B$ we find that~\eqref{1.4b} holds for 
$\alpha(r)= 6m\,\varepsilon_3(r/2)$.
It is easy to see that~\eqref{1.4a} holds for the functions $\beta$ and~$\alpha$ that we have specified.
It thus follows from Theorem~\ref{thm1.4} that $I(f)$ has positive measure.
In fact, by Lemma~\ref{lemma-zheng}, even $I(f)\cap J(f)$ has positive measure.
\end{proof}

The idea of the proof of Theorem~\ref{thm1.3} is that multiplying an entire 
function $f$ by a function of smaller growth does not change the asymptotics.
More precisely, we shall show that~\eqref{4d} and~\eqref{4h} remain valid
outside certain exceptional disks of small area.

In order to obtain the appropriate modification of~\eqref{4d} we use the following
result~\cite[Chapter~I, Theorem~11]{Levin1964}.
\begin{lemma} \label{levin-thm11}
Let $R>0$ and let $g$ be holomorphic in $\overline{D(0,2eR)}$ with $g(0)=1$.
Let $0<\eta<3e/2$.
Then there exist $l\in\N$,
$a_1,a_2,\dots,a_l\in\C$ and $s_1,s_2,\dots,s_l>0$ satisfying
\[
\sum_{k=1}^l s_k \leq 4\eta R
\]
such that
\[
\log|g(z)| > -\left(2+\log\frac{3e}{2\eta}\right) \log M(2eR,g)
\quad\text{for }
z\in \overline{D(0,R)}\setminus \bigcup_{k=1}^l D(a_k,s_k).
\]
\end{lemma}
It follows from the proof, which is based on Cartan's lemma, that $l$ can be chosen
to satisfy $l\leq n(2R,0)$.
Using that 
\begin{equation}\label{5a}
N(er,0)=\int_0^{er}n(t,0)\frac{dt}{t}\geq 
\int_{r}^{er}n(t,0)\frac{dt}{t}\geq 
n(r,0) \int_{r}^{er}\frac{dt}{t}=n(r,0)
\end{equation}
for $r>0$ we see, using Jensen's formula, that 
\begin{equation}\label{5b}
l\leq N(2eR,0)\leq M(2eR,g)
\end{equation}
in Lemma~\ref{levin-thm11}.

For the modification of~\eqref{4h} we shall need the following result due to Fuchs and
Macintyre~\cite{Fuchs1940}. 
\begin{lemma} \label{fuchs}
Let $z_1,z_2,\dots,z_n\in\C$ and $H>0$.
Then there exist $m\in\{1,\dots,n\}$,
$b_1,b_2,\dots,b_m\in\C$ and $t_1,t_2,\dots,t_m>0$ satisfying
\[
\sum_{k=1}^m t_k^2 \leq 4H^2
\]
such that
\[
\sum_{k=1}^m\frac{1}{|z-z_k|}\leq  \frac{2m}{H}
\quad\text{for}\
z\in\C\setminus \bigcup_{k=1}^m D(b_k,t_k).
\]
\end{lemma}

\begin{proof}[Proof of Theorem~\ref{thm1.3}]
We note that if $f$ satisfies the hypotheses of Theorem~\ref{thm1.2} for some $N\in\N$,
then it also satisfies them for every larger value of~$N$.
We may thus assume that $N>L$.
Put $\delta(r)=1/\log^L r$ and let $\varepsilon_1(r)$, $\varepsilon_2(r)$, $\varepsilon_3(r)$,
$U_j$, $V_j$, $W_j$ and $X_j$ be as in
the proof of Theorem~\ref{thm1.2} so that~\eqref{4d} and~\eqref{4h} hold. 
Actually, we may also assume that instead of~\eqref{4d} we have
\begin{equation}\label{4d1}
\log |f(z)|\geq 2|z|^{\rho(|z|)} \varepsilon_1(|z|)
\quad\text{for} \ z\in U_j.
\end{equation}
as this can be achieved by choosing a larger value of~$p$.

We may assume without loss of generality that $g(0)=1$.
For large $r>0$ we apply Lemma~\ref{levin-thm11} with $R=2r$ and $\eta=\delta(r)$
and conclude that there exist 
disks $D(a_1,s_1),\dots,D(a_l,s_l)$ with
\begin{equation}\label{4d2}
\sum_{k=1}^l s_k \leq 
8r\delta(r)
\end{equation}
such that
\begin{align*}
\log|g(z)| 
&> -\left(2+\log\frac{3e}{2} +\log\frac{1}{\delta(r)}\right) \log M(4er,g)
\\ &
\geq -O\!\left(r^{\rho(r)}\delta(r)\log\frac{1}{\delta(r)} \right)
\quad\text{for }
z\in \overline{D(0,R)}\setminus \bigcup_{k=1}^l D(a_k,s_k).
\end{align*}
Hence 
\begin{equation}\label{5c}
\log|g(z)| 
\geq -|z|^{\rho(|z|)}\sqrt{\delta(|z|)} 
\geq -|z|^{\rho(|z|)} \varepsilon_1(|z|) 
\quad\text{for }
z\in \ann(r)\setminus \bigcup_{k=1}^l D(a_k,s_k),
\end{equation}
provided $r$ is sufficiently large.
Combining this with~\eqref{4d1} we deduce that the product $F=fg$ satisfies
\begin{equation}\label{5d}
\log |F(z)|\geq |z|^{\rho(|z|)} \varepsilon_1(|z|)
\quad\text{for }
z\in U_j\cap \ann(r)\setminus \bigcup_{k=1}^l D(a_k,s_k),
\end{equation}
Here
\begin{equation}\label{5d1}
l\leq M(4er,g)=O\!\left(r^{\rho(r)}\delta(r)\right)
\end{equation}
by~\eqref{5b}.

To estimate the logarithmic derivative we note that 
for $s>|z|$ we have~\cite[p.~88]{Goldberg2008}
\begin{equation}\label{5e}
\left| \frac{g'(z)}{g(z)}\right|
\leq \frac{4s}{(s-|z|)^2} T(s,g)+
\sum_{|z_j|\leq s} \frac{2}{|z-z_j|},
\end{equation}
where $(z_j)$ is the sequence of zeros of~$g$.
Here we take $s=4r$ so that
\[
\frac{4s}{(s-|z|)^2}\leq \frac{16r}{(4r-2r)^2}=\frac{4}{r}
\quad\text{for }
|z|\leq 2r
\]
and hence
\begin{equation}\label{5f}
\frac{4s}{(s-|z|)^2} T(s,g) \leq \frac{4T(4r,g)}{r}
=O\!\left(r^{\rho(r)-1}\delta(r)\right)
\quad\text{for }
z\in \ann(r).
\end{equation}

To estimate the sum on the right side of~\eqref{5e} we apply Lemma~\ref{fuchs} with 
$H=r\sqrt{\delta(r)}/2$.
This yields disks $D(b_1,t_1),\dots,D(b_m,t_m)$ with 
\begin{equation}\label{5g}
\sum_{k=1}^m t_k^2 \leq r^2\delta(r)
\end{equation}
such that
\begin{equation}\label{5h}
\sum_{k=1}^m\frac{1}{|z-z_k|}\leq  \frac{4m}{r\sqrt{\delta(r)}}
\quad\text{for}\
z\in\C\setminus \bigcup_{k=1}^m D(b_k,t_k).
\end{equation}
Here
\begin{equation}\label{5i}
m\leq n(4r,0)\leq N(4er,0)\leq M(4er,g)
=O\!\left(r^{\rho(r)}\delta(r) \right)
\end{equation}
as in~\eqref{5a} and~\eqref{5b} and hence
\begin{equation}\label{5j}
\sum_{k=1}^m\frac{1}{|z-z_k|}
=O\!\left(r^{{\rho(r)}-1}\sqrt{\delta(r)} \right)
\quad\text{for }
z\in \ann(r)\setminus \bigcup_{k=1}^m D(b_k,t_k).
\end{equation}
Combining~\eqref{5e}, \eqref{5f} and \eqref{5j} and noting that $\sqrt{\delta(r)}=o(\varepsilon_2(r))$ we
conclude that
\begin{equation}\label{5k}
\left| \frac{zg'(z)}{g(z)}\right|
=o (r^{{\rho(r)}-1}\varepsilon_2(r))
\quad\text{for }
z\in \ann(r)\setminus \bigcup_{k=1}^m D(b_k,t_k).
\end{equation}
Noting that
\begin{equation}\label{5l}
 \frac{zF'(z)}{F(z)}
= \frac{zf'(z)}{f(z)}
+ \frac{zg'(z)}{g(z)}
\end{equation}
we deduce from~\eqref{4h} that 
\begin{equation}\label{5m}
\re \frac{zF'(z)}{F(z)} 
\geq \frac12 c_j \rho |z|^{\rho(|z|)}\varepsilon_2(|z|)
\quad\text{for} \ z\in W_j\cap 
 \ann(r)\setminus \bigcup_{k=1}^m D(b_k,t_k),
\end{equation}
provided $r$ is large enough.
Taking again $\beta(r)=\exp\!\left(r^{\rho(r)}\varepsilon_1(r)\right)$ 
we deduce from~\eqref{5d} and~\eqref{5m} that
\begin{equation}\label{5n}
A(F,\beta)\supset  W_j\cap \ann(r)\setminus 
\left( \bigcup_{k=1}^l D(a_k,s_k)\cup \bigcup_{k=1}^m D(b_k,t_k) \right)
\end{equation}
for $j=1,\dots,m$ and large~$r$.

By~\eqref{5m} there exists $C>0$ such that 
\begin{equation}\label{5o}
\frac{32|F(z)|}{|F'(z)|}\leq q_r:=C\frac{r^{1-{\rho(r)}}}{\varepsilon_2(r)}
\quad\text{for} \ z\in W_j\cap 
 \ann(r)\setminus \bigcup_{k=1}^m D(b_k,t_k).
\end{equation}
With $X=\bigcup_{j=1}^m X_j$ as in the proof of Theorem~\ref{thm1.2} we now conclude that
\begin{equation}\label{5p}
B(F,\beta)\supset  X\cap \ann(r)\setminus 
\left( \bigcup_{k=1}^l D(a_k,s_k+q_r)\cup \bigcup_{k=1}^m D(b_k,t_k+q_r) \right).
\end{equation}
Using~\eqref{4d2}, \eqref{5d1} and the definition of $q_r$ in~\eqref{5o} we see that
\begin{equation}\label{5q}
\begin{aligned}
\meas\!\left( \bigcup_{k=1}^l D(a_k,s_k+q_r)\right)
&\leq \pi\sum_{k=1}^l (s_k+q_r)^2
\leq 2\pi \sum_{k=1}^l (s_k^2+q_r^2)
\\ &
\leq 2 \pi\left(\sum_{k=1}^l s_k\right)^2+ 2\pi l q_r^2
\\ &
= O\!\left( r^2 \delta(r)^2 \right) + O\!\left( r^{2-\rho(r)} \frac{\delta(r)}{\varepsilon_2(r)^2}\right)
\\ &
= o\!\left( r^2 \varepsilon_3(r/2) \right) 
\end{aligned}
\end{equation}
for large~$r$.  Analogously, \eqref{5g} and \eqref{5i} yield
\begin{equation}\label{5r}
\meas\!\left( \bigcup_{k=1}^m D(b_k,t_k+q_r)\right)
= o\!\left( r^2 \varepsilon_3(r/2) \right) 
\end{equation}
for large~$r$.
It follows from~\eqref{5q} and~\eqref{5r} that
\begin{equation}\label{5s}
\dens\!\left( \bigcup_{k=1}^l D(a_k,s_k+q_r)\cup \bigcup_{k=1}^m D(b_k,t_k+q_r),\ann(r)\right)
= o\!\left( r^2 \varepsilon_3(r/2) \right) .
\end{equation}
Together with~\eqref{4i} and~\eqref{5p} we now see that
\begin{equation}\label{5t}
\dens(B(F,\beta),\ann(r))
\geq 1-7m\, \varepsilon_3(r/2)
\end{equation}
for large~$r$.
As in the proof of Theorem~\ref{thm1.2} the conclusion now 
follows directly from Theorem~\ref{thm1.4}.
\end{proof}

The theorem of Steinmetz~\cite[Theorem~1, Corollary~1]{Steinmetz1989} 
referred to in the introduction is the following.
\begin{lemma}\label{thm-steinmetz}
Let $w$ be an entire transcendental solution of~\eqref{lde}, with rational coefficients $p_0,\dots,p_{n-1}$.
Then there exist $q\in\N$, $C>0$ and $\theta_1,\dots,\theta_{m+1}\in\R$ with 
$\theta_1<\theta_2<\dots<\theta_m<\theta_{m+1}=\theta_1+2\pi$
such that the number $n(r)$ of zeros of $w$ of modulus at most $r$ which are contained 
in the union of the domains
\[
S_j=\left\{re^{i\theta}\colon r> 1\text{ and }
\theta_j+Cr^{-1/q} \log r<\theta<\theta_{j+1}-Cr^{-1/q}\log r \right\},
\quad  
\]
satisfies $n(r)=O(\log r)$.

Let $(z_k)$ be the sequence of zeros contained in the union of the $S_j$.
Then there exist $\varepsilon>0$ and, for $j=1,\dots,m$, a polynomial $P_j$ such that 
\begin{equation}\label{5t1}
\log|w(z)|=\re P_j(z^{1/q}) +O(\log|z|)  
\quad  \text{for}\ z\in S_j\setminus  \bigcup_k D(z_k,|z_k|^{1-\varepsilon}).
\end{equation}
\end{lemma}
The rays $\{z\colon \arg z=\theta_j\}$ are called \emph{Stokes rays} in the theory
of differential equations.
\begin{proof}[Proof of Theorem~\ref{thm1.1}]
Let $S_1,\dots,S_m$ and $(z_k)$ be as in Lemma~\ref{thm-steinmetz}.
Let $g$ be the canonical product formed with the~$z_k$; that is,
\begin{equation}\label{5u}
g(z)=\prod_{k=1}^\infty \left(1-\frac{z}{z_k}\right).
\end{equation}
Noting that $|z_k|\geq 1$ for all $k$ by the definition of the $S_j$ we deduce from
classical estimates~\cite[Section 2.3]{Goldberg2008} of canonical products that
\begin{equation}\label{5v}
\log M(r,g)\leq \int_1^r\frac{n(t)}{t}dt + r \int_r^\infty \frac{n(t)}{t^2}dt
=O\!\left((\log r)^2\right).
\end{equation}
Moreover, if $|z-z_k|\geq |z_k|^{1-\varepsilon}$ for all~$k$, then
\begin{equation}\label{5w}
|g(z)|= \prod_{|z_k|\leq 2|z|} \frac{|z_k-z|}{|z_k|} 
\cdot \prod_{|z_k|> 2|z|} \left| 1-\frac{z}{z_k}\right| 
\geq \prod_{|z_k|\leq 2|z|}  |z_k|^{-\varepsilon}
\cdot \prod_{|z_k|> 2|z|} \left( 1-\frac{|z|}{|z_k|}\right)
\end{equation}
and hence 
\begin{equation}\label{5x}
\log |g(z)|\geq -\varepsilon \sum_{|z_k|\leq 2|z|} \log|z_k|
- 2 |z| \sum_{|z_k|> 2|z|} \frac{1}{|z_k|}.
\end{equation}
Now 
\begin{equation}\label{5y}
\sum_{|z_k|\leq 2|z|} \log|z_k| \leq n(2|z|)\log (2|z|)=O\!\left((\log |z|)^2\right)
\end{equation}
and
\begin{equation}\label{5z}
|z| \sum_{|z_k|> 2|z|} \frac{1}{|z_k|}
=|z|\int_{2|z|}^\infty \frac{1}{t}dn(t)=|z|\int_{2|z|}^\infty \frac{n(t)}{t^2} dt
=O\!\left((\log |z|)^2\right)
\end{equation}
so that altogether we obtain 
\begin{equation}\label{6a}
|\log |g(z)||= O\!\left((\log |z|)^2\right)
\quad  
\text{for}\ z\notin \bigcup_k D(z_k,|z_k|^{1-\varepsilon}).
\end{equation}
Let $f=w/g$. It then follows from~\eqref{5t1} and~\eqref{6a} that
\begin{equation}\label{6b}
\log|f(z)|=\re P_j(z^{1/q}) +O\!\left((\log|z|)^2\right)
\quad  \text{for}\ z\in S_j\setminus  \bigcup_k D(z_k,|z_k|^{1-\varepsilon}).
\end{equation}
Some of the disks $D(z_k,|z_k|^{1-\varepsilon})$ may overlap. However, since
\begin{equation}\label{6c}
\sum_{|z_k|\leq 2|z|}|z_k|^{1-\varepsilon}\leq n(2|z|) (2|z|)^{1-\varepsilon}
=O\!\left(|z|^{1-\varepsilon}\log|z|\right),
\end{equation}
the components of $\bigcup_k D(z_k,|z_k|^{1-\varepsilon})$ that intersect the disk
$D(0,r)$ have diameter $O\!\left(r^{1-\varepsilon}\log r\right)$. 
Assuming $\varepsilon<1/q$ we deduce that there exists $C'>0$ such that
for sufficiently large $r_0$  every component of $\bigcup_k D(z_k,|z_k|^{1-\varepsilon})$
that intersects
\begin{equation}\label{6d}
T_j=\left\{re^{i\theta}\colon r\geq r_0\text{ and } \theta_j+C'r^{-\varepsilon}\log r\leq\theta\leq 
\theta_{j+1}-C'r^{-\varepsilon}\log r\right\}
\end{equation}
is contained in $S_j$.
Since $\log|f(z)|$ and $\re P_j(z^{1/q})$ are harmonic in these components and since
$\log|f(z)|=\re P_j(z^{1/q}) +O\!\left((\log|z|)^2\right)$
on their boundary by~\eqref{6b}, this implies 
that this equation actually also holds inside these components.
We deduce that
\begin{equation}\label{6e}
\log|f(z)|=\re P_j(z^{1/q}) +O\!\left((\log|z|)^2\right)
\quad  \text{for}\ z\in T_j.
\end{equation}
Thus $f$ satisfies the hypothesis of Theorem~\ref{thm1.2}.
Since $w=fg$ it now follows from Theorem~\ref{thm1.3}
that $I(w)\cap J(w)$ has positive measure.
\end{proof}

\section{Proof of Theorem~\ref{thm1.5}}
We may assume that $c=1$ and that $f$ is given as a
Weierstra{\ss} product
\begin{equation}\label{9a}
f(z)=\prod_{k=1}^\infty E\!\left(\frac{z}{a_k}, p\right)
\end{equation}
with $|a_k|>\exp^L(1)$  and $|\arg a_k|\leq \varepsilon(|a_k|)$ for all $k\in \N$.
We put $\varepsilon(r)=1$ for $r<\exp^L(1)$.
We consider the auxiliary function
\begin{equation}\label{9b}
f^*(z)=\prod_{k=1}^\infty E\!\left(\frac{z}{|a_k|}, p\right).
\end{equation}
Our arguments are similar to those in \cite[Chapter~2, Sections~2 and~5]{Goldberg2008}.
We have
\begin{equation}\label{9c}
\log f^*(z)= -z^{p+1} \int_0^\infty \frac{n(t, 0)}{t^{p+1} (t-z)}dt
\end{equation}
and
\begin{equation}\label{9d}
I(z):=\int_0^\infty \frac{t^{\rho(r)} dt}{t^{p+1}(t-z)}
=-\frac{\pi e^{-i \pi (\rho(r)-p)}}{\sin \pi (\rho(r)-p)} z^{\rho(r)-p-1},
\end{equation}
with logarithms and powers defined for $0<\arg z<2\pi$.
For $z=re^{i\theta}$ thus
\begin{equation} \label{9e}
\re\!\left( z^{p+1} I(z)\right)=-\frac{\pi\cos((\theta-\pi)\rho(r))}{\sin \pi \rho(r)} r^{\rho(r)}.
\end{equation}
By~\eqref{9c} and \eqref{9d} 
we have
\begin{equation} \label{9f}
\begin{aligned}
\left|\log |f^*(z)| + \re (z^{p+1} I(z))\right|
 =
\left| \re \left( z^{p+1} \int_0^\infty \frac{t^{\rho(r)}-n(t, 0)}{t^{p+1} (t-z)}dt\right)\right|.
\end{aligned}
\end{equation}
We define 
\[
b(r)=\frac{1}{a(r)}=\exp\!\left(\varepsilon(r)^{-1/4}\right).
\]
As in the proof of Theorem 2.2 in \cite[Chapter 2, Section 2]{Goldberg2008} we see
that for every $k\in [a(r), b(r)]$ there exists $\kappa\in [a(r), b(r)]$ such that 
\begin{equation}\label{rho(kr)}
  \begin{aligned}
    |\rho(kr)-\rho(r)|= & \kappa r|\rho'(\kappa r)||\log k| \\
    \le  & \varepsilon (\kappa r) \frac{\log b(r)}{\log (a(r) r)}=(1+o(1)) \frac{\varepsilon(r)^{3/4}}{\log r}
 \quad\text{as}\  r\to\infty.
      \end{aligned}
\end{equation}
Thus for all such $k$ we have 
\begin{equation}\label{e:prox_or_asym}
r^{\rho(kr)-\rho(r)}=\exp\!\left((1+o(1))\varepsilon(r)^{3/4}\right) \to 1
 \quad\text{as}\ r\to \infty.
\end{equation}
We choose $\delta>0$  such that $\min\{ \rho(t)-p,  p+1-\rho(t)\}>\delta$ for large $t$.

In order to estimate  the integral in~\eqref{9f} we follow the arguments from \cite[Chapter 2, Section 5]{Goldberg2008} 
and split the integral as follows:
\begin{equation} \label{9fest}
\begin{aligned}
  \int_0^\infty \frac{n(t, 0)-t^{\rho(r)}}{t^{p+1} (t-z)}dt
= & \int_0^{a(r)r} \frac{n(t, 0)}{t^{p+1} (t-z)}dt - \int_0^{a(r)r} \frac{t^{\rho(r)}}{t^{p+1} (t-z)}dt \\ 
 & +\int_{a(r)r}^{b(r)r} \frac{n(t, 0)-t^{\rho(t)}}{t^{p+1} (t-z)}dt+ \int_{a(r)r}^{b(r)r} \frac{t^{\rho(t)}-t^{\rho(r)}}{t^{p+1} (t-z)}dt
\\  & +   \int_{b(r)r}^\infty \frac{n(t, 0)}{t^{p+1} (t-z)}dt -\int_{b(r)r}^\infty \frac{t^{\rho(r)}}{t^{p+1} (t-z)}dt  .
\end{aligned}
\end{equation}
For the first two integrals on the right-hand side of \eqref{9fest} we find, using  standard properties of the 
proximate order and \eqref{e:prox_or_asym}, that there exists a constant $C$ such that
\begin{equation}\label{e:1-2}
\begin{aligned}
\left|  \int_0^{a(r)r} \frac{n(t, 0)}{t^{p+1} (t-z)}dt - \int_0^{a(r)r} \frac{t^{\rho(r)}}{t^{p+1} (t-z)}dt  \right|
& \le \frac{C}{r}  \int_1^{a(r)r} \frac{t^{\rho(t)} + t^{\rho(r)}}{t^{p+1} }dt\\
& \leq Ca(r)^\delta r^{\rho(r)-p-1} 
\\ & 
= o\!\left(\varepsilon(r) r^{\rho(r)-p-1}\right)
\end{aligned}
\end{equation}
as $r\to\infty$.
Similarly,
\begin{equation}\label{e:5-6}
\begin{aligned}
\left|  \int_{b(r)r}^\infty \frac{n(t, 0)}{t^{p+1} (t-z)}dt -\int_{b(r)r}^\infty \frac{t^{\rho(r)}}{t^{p+1} (t-z)}dt  \right| 
& \le 
C  \int_{b(r)r}^\infty  \frac{t^{\rho(t)} + t^{\rho(r)}}{t^{p+2} }dt\\
& \le \frac{C r^{\rho(r)-p-1}}{b(r)^\delta} 
\\ & 
= o({{\varepsilon(r)}}  r^{\rho(r)-p-1}).
\end{aligned}
\end{equation}
The hypothesis \eqref{e:asymp_ray} and  \eqref{e:prox_or_asym} allow us to estimate 
the third and the forth integral from \eqref{9fest} as follows:
\begin{equation} \label{e:3-4a}
\begin{aligned}
& \quad 
\left| \int_{a(r)r}^{b(r)r} \frac{n(t, 0)-t^{\rho(t)}}{t^{p+1} (t-z)}dt
+ \int_{a(r)r}^{b(r)r} \frac{t^{\rho(t)}-t^{\rho(r)}}{t^{p+1} (t-z)}dt\right|
\\ &
\le  \frac{C}{\sin \frac{\sqrt{\varepsilon(r)}}{2}} 
 \int_{a(r)r}^{b(r)r} \frac{t^{\rho(t)}\varepsilon(t)^{3/4}}{t^{p+1}(t+r) }dt
\\ &
\le (2C+o(1))\sqrt[4]{\varepsilon(r)} \int_{a(r)r}^{b(r)r} \frac{t^{\rho(r)}}{t^{p+1}(t+r) }dt.
\end{aligned}
\end{equation}
Since 
\begin{equation} \label{e:3-4b}
\begin{aligned}
\int_{a(r)r}^{b(r)r} \frac{t^{\rho(r)}}{t^{p+1}(t+r) }dt
& = r^{\rho(r)-p-1} \int_{a(r)}^{b(r)} \frac{\tau ^{\rho( r)}}{\tau^{p+1}(1+\tau) }d\tau 
\\ &
\leq r^{\rho(r)-p-1} \int_{0}^{\infty} \frac{\tau ^{\rho( r)}}{\tau^{p+1}(1+\tau) }d\tau 
\end{aligned}
\end{equation}
this yields
\begin{equation} \label{e:3-4}
\left| \int_{a(r)r}^{b(r)r} \frac{n(t, 0)-t^{\rho(t)}}{t^{p+1} (t-z)}dt
+ \int_{a(r)r}^{b(r)r} \frac{t^{\rho(t)}-t^{\rho(r)}}{t^{p+1} (t-z)}dt\right|
=O\left({\sqrt[4]{\varepsilon(r)}}  r^{\rho(r)-p-1}\right).
\end{equation}
Combining \eqref{9fest}, \eqref{e:1-2}, \eqref{e:5-6} and \eqref{e:3-4} with~\eqref{9f} we find for $z=re^{i\theta}$ with
$\theta$ satisfying~\eqref{e:0b} that
\begin{equation} \label{9l}
\left|\log |f^*(z)| + \re (z^{p+1} I(z))\right|
=O\!\left(\sqrt[4]{\varepsilon(r)}r^{\rho(r)}\right).
\end{equation}

Next we note that
\begin{equation} \label{9m}
\begin{aligned}
\log \left|E\!\left(\frac{z}{|a_k|} ,p\right)\right|
-\log \left|E\!\left(\frac{z}{a_k} ,p\right)\right|
&=
\re\left(\log E\!\left(\frac{z}{|a_k|} ,p\right)
-\log E\!\left(\frac{z}{a_k} ,p\right)\right)
\\ &
=
\re \int _{a_k}^{|a_k|} \frac{d}{du} \log E\!\left(\frac{z}{u} ,p\right) du
\\ &
=
\re \int _{a_k}^{|a_k|}
\frac{z^{p+1}}{u^{p+1} (u-z)} du
\end{aligned}
\end{equation}
and thus, integrating over the appropriate part of the circle with radius $|a_k|$,
and noting that $|u-z|\geq ||a_k|-z|/2$ for $u$ on that part of the circle
and $z=re^{i\theta}$ with $\theta$ satisfying~\eqref{e:0b},
we see that
\begin{equation} \label{9n}
\begin{aligned}
\left|\log \left|E\!\left(\frac{z}{|a_k|} ,p\right)\right|
-\log \left|E\!\left(\frac{z}{a_k} ,p\right)\right|\right|
&\leq
\int _{a_k}^{|a_k|} \frac{r^{p+1}}{|a_k|^{p+1} |u-z|} |du|
\\ &
\leq 2 |a_k|\!\cdot\! |\arg a_k| \frac{r^{p+1}}{|a_k|^{p+1} \left||a_k|-z\right|}
\\ &
\leq  2 \frac{\varepsilon(|a_k|) r^{p+1}}{|a_k|^{p} \left||a_k|-z\right|},
\end{aligned}
\end{equation}
provided $r$ is sufficiently large.
It follows that
\begin{equation} \label{9p}
\begin{aligned}
\left|\log |f^*(z)| - \log |f(z)| \right|
&\leq
\sum_{k=1}^\infty
\left|\log \left|E\!\left(\frac{z}{|a_k|} ,p\right)\right|
-\log \left|E\!\left(\frac{z}{a_k} ,p\right)\right|\right|
\\ &
\leq 2r^{p+1} \int_0^\infty   \frac{\varepsilon(t)}{t^{p} \left|t-z\right|} dn(t,0).
\end{aligned}
\end{equation}
Again we split the integral into three parts:
Noting that
$\varepsilon'(t)=o(\varepsilon(t)/t)$
we obtain
\begin{equation} \label{9r}
\begin{aligned}
r^{p+1} \int_0^{r/2}   \frac{\varepsilon(t)}{t^{p} \left|t-z\right|} dn(t,0)
&
\leq
2r^p \int_0^{r/2} \frac{\varepsilon(t)}{t^{p}}dn(t,0)
\\ &
=2^{p+1}\varepsilon(r/2) n(r/2,0)
\\ & \qquad
+(2p+o(1)) r^p \int_0^{r/2} \frac{\varepsilon(t)n(t,0)}{t^{p+1}}dt
\\ &
\leq O\!\left(\varepsilon(r)r^{\rho(r)}\right) +(2p+o(1)) r^p \int_0^{r/2}  t^{\rho_1(t)-p-1}dt
\\ &
=O\!\left(r^{\rho_1(r)}\right)
\end{aligned}
\end{equation}
where 
\begin{equation*} \label{9g}
\rho_1(t)=\rho(t)+\frac{\log\varepsilon(t)}{\log t}=\rho(t)-\frac{\log^{L+1} t}{\log t}.
\end{equation*}
Similarly,
\begin{equation} \label{9s}
\begin{aligned}
r^{p+1} \int_{2r}^{\infty}   \frac{\varepsilon(t)}{t^{p} \left|t-z\right|} dn(t,0)
& \leq
2 r^{p+1}\int_{2r}^{\infty}   \frac{\varepsilon(t)}{t^{p+1}} dn(t,0)
\\ &
\leq 2(p+1+o(1))r^{p+1} \int_{2r}^{\infty} \frac{\varepsilon(t)n(t,0)}{t^{p+2}} dt
\\ &
= 2(p+1+o(1))r^{p+1} \int_{2r}^{\infty} t^{\rho_1(t)-p-2}dt
\\ &
=O\!\left(r^{\rho_1(r)}\right).
\end{aligned}
\end{equation}
We also have
\begin{equation} \label{9t}
r^{p+1} \int_{r/2}^{2r}   \frac{\varepsilon(t)}{t^{p} \left|t-z\right|} dn(t,0)
\leq \frac{2^p\varepsilon(r/2)}{r^{p}}\frac{\pi}{2r\sqrt{\varepsilon(r)}} n(2r,0)
=O\!\left(r^{\rho_2(r)}\right)
\end{equation}
for $z=re^{i\theta}$ with $\theta$ satisfying~\eqref{e:0b}, where
\begin{equation*} \label{9k}
\rho_2(t)=\rho_1(t)+\frac{\log\sqrt{\varepsilon(t)}}{\log t}=\rho(t)-\frac12 \frac{\log^{L+1} t}{\log t}, 
\end{equation*}
provided $r$ is sufficiently large.
Combining the last three estimates with~\eqref{9p} we find that
\begin{equation} \label{9u}
\left|\log |f^*(z)| - \log |f(z)| \right|
=O\!\left(r^{\rho_2(r)}\right)
=O\!\left(\sqrt[4]{\varepsilon(r)}r^{\rho(r)}\right)
\end{equation}
for such~$z$.
Together with~\eqref{9e} and~\eqref{9l} this yields the conclusion.

\section{Further results}\label{further}
It is not clear to which extent the condition in Theorems~\ref{thm1.2} and~\ref{thm1.3}
that all or most zeros are distributed along finitely many rays can be omitted.
Suppose that $f$ has completely regular growth; that is, 
\eqref{C0} and~\eqref{def-crg} are satisfied.
The results  by Gol'dberg and Korenkov~\cite{Goldberg1978,Goldberg1981}
as well as Gol'dberg, Sodin and Strochik~\cite{Goldberg1992} that were mentioned
in the introduction say in particular that 
for $1<\mu\leq 2$ there exist disks $D(b_k,t_k)$ satisfying
\begin{equation}\label{7a}
\sum_{|b_k|\leq r}t_k^\mu =o(r^\mu) 
\end{equation}
such that
\begin{equation}\label{7a1}
\re\!\left(\frac{zf'(z)}{f(z)}\right) =\rho h(\theta) r^{\rho(r)} +o(r^{\rho(r)})
\quad\text{for}\ 
z=re^{i\theta}\notin \bigcup_k D(b_k,t_k).
\end{equation}
In particular, choosing $\mu=2$ we can achieve that
\begin{equation}\label{7b}
\sum_{|b_k|\leq r}t_k^2 =o(r^2) .
\end{equation}
Also, it follows from~\eqref{C0} that the exceptional disks $D(a_k,s_k)$ occurring in~\eqref{def-crg}
satisfy 
\begin{equation}\label{7c}
\sum_{|a_k|\leq r}s_k^2 =o(r^2) .
\end{equation}
However, in order to apply Theorem~\ref{thm1.4} we need a quantitative estimate instead of the term
$o(r^2)$ in~\eqref{7b} and~\eqref{7c}.

For the following two results we also note that the disks $D(a_k,s_k)$ and $D(b_k,t_k)$ contain zeros of $f$
and thus 
\begin{equation}\label{7d}
n(r,(a_k)):=\sum_{|a_k|\leq r} 1 =O(r^{\rho(r)})
\quad  \text{and}\quad 
n(r,(b_k))=O(r^{\rho(r)}).
\end{equation}

\begin{theorem}\label{thm4.1}
Let $f$ be an entire function of completely regular growth with respect
to the proximate order $\rho(r)$ and suppose that
the indicator is strictly positive; that is, we have~\eqref{def-crg} for some function $h$
satisfying $h(\theta)>0$ for all $\theta\in\R$.
Let $\varepsilon(r)=1/\log^N r$ for some $N\in\N$ 
and suppose that the radii $s_k$ and $t_k$ of the exceptional disks in~\eqref{def-crg} 
and~\eqref{7a1} satisfy
\begin{equation}\label{7e}
\sum_{|a_k|\leq r}s_k^2 =O\!\left(r^2\varepsilon(r)\right) 
\quad  \text{and}\quad 
\sum_{|b_k|\leq r}t_k^2 =O\!\left(r^2\varepsilon(r)\right).
\end{equation}
Then $I(f)\cap J(f)$ has positive measure.
\end{theorem}
\begin{theorem}\label{thm4.2}
Let $f$ be an entire function of completely regular growth with respect
to the proximate order $\rho(r)$ and suppose that
the indicator is positive except for isolated points.
Let $\varepsilon(r)=1/\log^N r$ for some $N\in\N$ 
and suppose that there exist disks $D(a_k,s_k)$ and $D(b_k,t_k)$ with
\begin{equation}\label{7f}
\log|f(re^{i\theta})|=h(\theta)r^{\rho(r)} +O\!\left(r^{\rho(r)}\varepsilon(r)\right)
\quad\text{for}\ 
z=re^{i\theta}\notin \bigcup_k D(a_k,s_k)
\end{equation}
and
\begin{equation}\label{7g}
\re\!\left(\frac{zf'(z)}{f(z)}\right) =\rho h(\theta) r^{\rho(r)} + O\!\left(r^{\rho(r)}\varepsilon(r)\right)
\quad\text{for}\ 
z=re^{i\theta}\notin \bigcup_k D(b_k,t_k),
\end{equation}
where the radii $s_k$ and $t_k$ of the exceptional disks in~\eqref{7f} 
and~\eqref{7g} satisfy~\eqref{7e}.
Then $I(f)\cap J(f)$ has positive measure.
\end{theorem}

\begin{proof}[Proof of Theorem~\ref{thm4.1}]
Let $0<c<\min_{\theta\in[0,2\pi]}h(\theta)$ and put $\beta(r)=\exp(cr^{\rho(r)})$.
Then
\begin{equation}\label{7h}
A(f,\beta) \supset  \{z\colon |z|>r_0\} \setminus 
\left( \bigcup_{k} D(a_k,s_k)\cup \bigcup_{k} D(b_k,t_k) \right)
\end{equation}
for large $r_0$ by~\eqref{def-crg} and~\eqref{7a1}.
For $z\notin \bigcup_{k} D(a_k,s_k)\cup \bigcup_{k} D(b_k,t_k)$ we also have 
\begin{equation}\label{7i}
\frac{32|f(z)|}{|f'(z)|} \leq \frac{32|z|}{\displaystyle \re\!\left(\frac{zf'(z)}{f(z)}\right)}
\leq \frac{32}{c \rho} |z|^{1-\rho(|z|)}
\end{equation}
by~\eqref{7a1} if $|z|$ is large.
This implies that there exists $C>0$ such that with $q_r=Cr^{1-\rho(r)}$ we have
\begin{equation}\label{7j}
B(f,\beta) \supset  \ann(r)\setminus 
\left( \bigcup_{|a_k|\leq 3r} D(a_k,s_k+q_r)\cup \bigcup_{|b_k|\leq 3r} D(b_k,t_k+q_r) \right)
\end{equation}
for large~$r$.
As in~\eqref{5q} and~\eqref{5r} we have
\begin{equation}\label{7k}
\begin{aligned}
\meas\!\left( \bigcup_{|a_k|\leq 3r} D(a_k,s_k+q_r) \right)
&
\leq 2\pi \sum_{|a_k|\leq 3r} (s_k^2+q_r^2)
\\ &
\leq 2 \pi\sum_{|a_k|\leq 3r} s_k^2+ 2\pi q_r^2 n(3r,(a_k))
\\ &
=O\!\left(r^2\varepsilon(r)\right)+O\!\left(r^{2-\rho(r)}\right)
\\ &
=O\!\left(r^2\varepsilon(r)\right)
\end{aligned}
\end{equation}
by~\eqref{7d} and~\eqref{7e}.
Analogously,
\begin{equation}\label{7l}
\meas\!\left( \bigcup_{|a_k|\leq 3r} D(a_k,s_k+q_r) \right)
=O\!\left(r^2\varepsilon(r)\right).
\end{equation}
It follows from~\eqref{7j}, \eqref{7k}  and~\eqref{7l} that there exists 
a constant $K>0$ such that
\begin{equation}\label{7m}
\dens(B(f,\beta),\ann(r))
\geq 1-K\varepsilon(r).
\end{equation}
As in the proof of Theorem~\ref{thm1.2} the conclusion now follows from
Theorem~\ref{thm1.4}.
\end{proof}

\begin{proof}[Proof of Theorem~\ref{thm4.2}]
We only sketch the argument, as it is largely analogous to those in the previous proofs.
Assume that $\theta_1<\theta_2<\dots <\theta_m<\theta_{m+1}=\theta_1+2\pi$ and
that $h(\theta)>0$ for $\theta_j<\theta<\theta_{j+1}$ and $j=1,\dots,m$.
In addition to the disks $D(a_k,s_k)$ and $D(b_k,t_k)$, we consider disks $D(c_k,u_k)$ 
which cover the regions
$\{re^{i\theta}\colon  r>1, \theta_j-\varepsilon(r) <\theta < \theta_j+\varepsilon(r)\}$.
These disks may be chosen such that the radii satisfy
\begin{equation}\label{7n}
\sum_{|c_k|\leq r}u_k^2 =O\!\left(r^2\varepsilon(r)\right).
\end{equation}
Similarly as in the previous proofs we now deduce from~\eqref{7f} 
and~\eqref{7g} that, for $z$ outside the disks 
$D(a_k,s_k)$, $D(b_k,t_k)$ and $D(c_k,u_k)$, we have
\begin{equation}\label{7o}
\log |f(z)|\geq c |z|^\rho \varepsilon(|z|)
\quad  \text{and}\quad 
\re\!\left(\frac{zf'(z)}{f(z)}\right) \geq c |z|^\rho \varepsilon(|z|).
\end{equation}
Putting again $\beta(r)=\exp(r^{\rho(r)}\varepsilon(r))$
and $q_r:=Cr^{1-\rho(r)}(\log r)^\varepsilon$ for some large constant $C$
we now find 
with
\begin{equation}\label{7q}
E(r)= \bigcup_{|a_k|\leq 3r} D(a_k,s_k+q_r)\cup \bigcup_{|b_k|\leq 3r} D(b_k,t_k+q_r) 
\cup \bigcup_{|c_k|\leq 3r} D(c_k,u_k+q_r)
\end{equation}
that
\begin{equation}\label{7r}
B(f,\beta) \supset  \ann(r)\setminus E(r)
\end{equation}
and hence
\begin{equation}\label{7s}
\dens(B(f,\beta),\ann(r)) \geq 1-K\varepsilon(r)
\end{equation}
for some constant~$K$, provided $r$ is large.
The conclusion then follows using the same arguments as before.
\end{proof}

\begin{remark}\label{rem1}
Suppose that $f$ has completely regular growth and that there exists $L\in\N$ 
and a $C_0$-set $E$ such that 
\begin{equation}\label{7u}
\log|f(re^{i\theta})|=h(\theta)r^{\rho(r)} +O\!\left(r^{\rho(r)}/\log^L r\right)
\quad\text{for}\ re^{i\theta}\notin E.
\end{equation}
We do not know whether this implies that 
the $C_0$-set can be chosen such that the corresponding radii~$s_k$ satisfy~\eqref{7e}
with $\varepsilon(r)=1/\log^N r$ for some $N\in\N$.
We also do not know whether we then have~\eqref{7g} with disks $D(b_k,t_k)$ satisfying~\eqref{7e}. 
In other words, it seems possible that the hypotheses of Theorem~\ref{thm4.2} are satisfied
if~\eqref{7u} holds with some $C_0$-set $E$ and if
$h$ is positive except for isolated points.
This would be a substantial generalization of Theorem~\ref{thm1.2}.
\end{remark}
\begin{acknowledgement}
We thank Weiwei Cui and the referee for helpful comments.
\end{acknowledgement}

\end{document}